\documentclass[12pt]{article}
\usepackage[margin=1in]{geometry}
\usepackage{amsmath,amsthm,amsfonts,amssymb,amscd,cite} 
\numberwithin{equation}{section}

\newtheorem{theorem}{Theorem}[section]
\newtheorem{lemma}[theorem]{Lemma}
\newtheorem{claim}[theorem]{Claim}

\newtheorem{conjecture}[theorem]{Conjecture}

\theoremstyle{remark}

\newcommand{\floor}[1]{\left\lfloor #1 \right\rfloor}

\title{The number of cycles of a given length in dense hamiltonian graphs: proving Hilton's conjecture}

\author{Chengli Li$^a$\footnote{Email: lichengli0130@126.com}, Leyou Xu$^b$\footnote{Email: leyouxu@m.scnu.edu.cn}, Bo Zhou$^b$\footnote{Email: zhoubo@m.scnu.edu.cn}\\ 
$^a$Department of Mathematics, East China Normal University,\\ Shanghai, 200241, P.R. China\\
$^b$School of Mathematical Sciences, South China Normal University, \\
Guangzhou 510631, P.R. China}

\date{}

\begin{document}
\maketitle

\begin{abstract}
A classical theorem of Sheehan in 1977 states that every hamiltonian graph $G$ of order $n$ satisfying
$e(G)>\left\lfloor \frac{n^2}{4}\right\rfloor+1$
contains at least two cycles of every length $\ell$, $3\le \ell\le n$. In the same paper, Sheehan recorded a conjecture of Hilton, which strengthens this conclusion by asserting that such a graph contains at least $n-\ell+2$ cycles of length $\ell$ for each $3\le \ell\le n$. We prove Hilton's conjecture for all hamiltonian graphs of order at least $440$.
\end{abstract}

\section{Introduction}
A graph is hamiltonian if it contains a Hamilton cycle (a cycle that goes through every vertex of the graph). Hamiltonicity has been a central topic in graph theory for many decades. Since deciding whether a graph is hamiltonian is NP-complete, a major line of research has sought structural conditions that guarantee Hamiltonicity. A classical starting point is Dirac's theorem~\cite{Dirac}: every $n$-vertex graph with minimum degree at least $\frac{n}{2}$ is hamiltonian. This theorem has inspired many refinements and variants concerning degree conditions, connectivity, random graphs, decompositions, robustness and related aspects of Hamiltonicity, see, for example,~\cite{ChvatalErdos,Csaba2016,Cuckler2009,Krivelevich2011,Krivelevich2014,Posa1976} and the survey~\cite{Gould}.

Pancyclicity is a natural strengthening of Hamiltonicity. An $n$-vertex graph is called pancyclic if it contains cycles of all lengths from $3$ to $n$. Bondy's well-known ``meta-conjecture''~\cite{Bondy1975} in 1973 suggests that, apart from simple extremal exceptions, any non-trivial  sufficient conditions for Hamiltonicity should in fact force pancyclicity. This conjecture is illustrated by Bondy's theorem~\cite{Bondy}, which strengthens Dirac's theorem by showing that every $n$-vertex graph with minimum degree at least $\frac{n}{2}$ is either pancyclic or isomorphic to $K_{\frac{n}{2},\frac{n}{2}}$. Further evidence was given by Bauer and Schmeichel~\cite{BS}, building on work of Schmeichel and Hakimi~\cite{SchmeichelHakimi}, who showed that the hamiltonian conditions of Bondy~\cite{Bondy1980}, Chv\'atal~\cite{Chvatal} and Fan~\cite{Fan} also imply pancyclicity  with certain exceptional graphs.

This paper considers a counting version of this circle of questions. Instead of asking merely whether each length occurs, we ask how many cycles of a prescribed length must appear in a dense hamiltonian graph. This is a natural strengthening.  For a graph $G$, let $e(G)$ denote its number of edges. Sheehan~\cite{Sheehan} proved in 1977 that if $G$ is a hamiltonian graph of order $n$ satisfying
\begin{equation}\label{eq1}
  e(G) > \floor{\frac{n^2}{4}} + 1,
\end{equation}
then $G$ contains at least two cycles of every length $\ell$, $3\le \ell\le n$. 
 Such  graphs are the dense graphs we study in this paper.  Let $c_\ell(G)$ be the number (or multiplicity) of cycles of length $\ell$ in $G$. In the same paper, Sheehan recorded the following conjecture due to  A.J.W. Hilton, which proposes a much sharper uniform multiplicity form of pancyclicity.

\begin{conjecture}[Hilton]\label{conj:Hilton}
Let $G$ be a hamiltonian graph of order $n$ satisfying \eqref{eq1}. Then, for every integer $\ell$ with $3\le \ell\le n$, $c_\ell(G) \ge n-\ell+2$.
\end{conjecture}

The conjecture is consistent with the extremal behavior at the two ends of the cycle spectrum (the set of cycle lengths of a graph): for $\ell=n$ it asks for two Hamilton cycles, a result already proved by Sheehan~\cite{Sheehan}, while for $\ell=3$ it follows from a result of Erd\H{o}s~\cite{Er55}, where it was proved that 
 every graph of order $n$ with  $e(G)=\floor{\frac{n^2}{4}} + \ell$ with $\ell=1,2,3$ contains at least $\ell\floor{\frac{n}{2}}$ triangles. 
Despite these endpoint cases, the full conjecture has remained open since Sheehan first posed it in 1977. Our main result confirms this %long-standing 
conjecture for  sufficiently large order.

\begin{theorem}\label{thm:main440}
Let $G$ be a hamiltonian graph of order $n\ge 440$ satisfying \eqref{eq1}.
Then, for every integer $\ell$ with $3\le \ell\le n$, $c_\ell(G) \ge n-\ell+2$.
\end{theorem}

In fact, we will prove a slightly stronger statement in Theorem~\ref{thmm}, which leaves only finite exceptional ranges for the lengths $6\le \ell\le 10$ and confirms Conjecture~\ref{conj:Hilton} for all other lengths.

%The proof is based on a deletion reduction. If Hilton's bound fails for some length $\ell<n$, then, after fixing a Hamilton cycle, we delete one edge outside the Hamilton cycle from each $\ell$-cycle. This produces a spanning hamiltonian $C_\ell$-free subgraph which still has many edges. The bipartite case is ruled out by a counting argument, while the non-bipartite case is handled by weak pancyclicity and extremal estimates for even and odd cycles.

We mention that this study also belongs to the problem to determine the number of substructures in a fixed class of
graphs, which received much attention, see \cite{AS,ErdosCircuits, Erd1}. For example, Erd\H{o}s \cite{Erd1} conjectured that the
maximum number of cycles of length $5$ in an $n$-vertex triangle-free graph is $\left(\frac{n}{5}\right)^5$. Through the work \cite{Ge}, it was confirmed independently in \cite{Gr,HHK}.

\section{Preliminaries}
All graphs in this paper are finite and simple. For a graph $G$, we write $V(G)$ and $E(G)$ for its vertex set and edge set, $e(G)=|E(G)|$ for its number of edges, and $|V(G)|$ for its order, usually denoted by $n$. A cycle of length $\ell$ (or an $\ell$-cycle) is denoted by $C_\ell$. We write $K_{a,b}$ for the complete bipartite graph with part sizes $a$ and $b$. For $X\subseteq V(G)$, $G[X]$ denotes the subgraph induced by $X$. For disjoint   $X, Y\subseteq V(G)$, $E_G(X,Y)$ denotes the set of edges of $G$ with one endpoint in $X$ and the other in $Y$. For an edge subset $E_1$ of the complement of $G$, $G+E_1$ denotes the graph with vertex set $V(G)$ and edge set $E(G)\cup E_1$. 
The neighborhood of a vertex $v$ in $G$ is denoted by $N_G(v)$. We use $H\subseteq G$ to mean that $H$ is a subgraph of $G$. Given a graph $F$, a graph $G$ is $F$-free if it contains no subgraph that is isomorphic to $F$. 
The extremal numbers $\operatorname{ex}(n,F)$ and $\operatorname{ex}(a,b,F)$ denote, respectively, the maximum number of edges in an $n$-vertex $F$-free graph and in an $F$-free bipartite graph with parts of  sizes $a$ and $b$. As usual,  $(x)_j$ denotes the falling factorial $x(x-1)\cdots (x-j+1)$. 

The circumference of a graph $G$ is the length of a longest cycle in $G$ and the girth of $G$ is the length of a shortest cycle in $G.$ A graph is called weakly pancyclic if it contains cycles of all lengths between its girth and circumference.

We first record two endpoint results.

\begin{lemma}[Sheehan~\cite{Sheehan}]\label{lem:SheehanErdos}
Let $G$ be a hamiltonian graph of order $n$ satisfying \eqref{eq1}. Then $G$ contains at least two Hamilton cycles.
\end{lemma}

\begin{lemma}[Erd\H{o}s~\cite{Er55}]\label{lem:Erdos}
Let $G$ be a graph of order $n$ satisfying \eqref{eq1}. Then $c_3(G) \ge n-1$. 
\end{lemma}

The next lemma is a classical pancyclicity theorem near the same density threshold.

\begin{lemma}[H\"aggkvist, Faudree, Schelp~\cite{HFS}]\label{lem:HFS}
Let $G$ be a hamiltonian graph of order $n$. If $e(G) \ge \floor{\frac{(n-1)^2}{4}} + 2$, then $G$ is pancyclic or bipartite.
\end{lemma}

We shall also use weak pancyclicity in dense non-bipartite graphs. Bollob\'as and Thomason~\cite{BollobasThomason} proved that every graph of order $n$ and size at least $\floor{\frac{n^2}{4}}-n+59$ is weakly pancyclic or bipartite. In the hamiltonian non-bipartite setting needed here, He and Hu obtained the following stronger bound.

\begin{lemma}[He, Hu~\cite{HeHu}]\label{lem:HeHu}
Let $G$ be a hamiltonian non-bipartite graph of order $n$. If $e(G) \ge \floor{\frac{n^2}{4}} - n + 12$, then $G$ is weakly pancyclic.
\end{lemma}

%For the bipartite part  setting, we use the following even-cycle result. Indeed, Entringer and Schmeichel \cite{EntringerSchmeichel} proved that for 
%$m>3$, every hamiltonian bipartite graph of order 
%$2m$ with $e(G)>\frac{m^2}{2}$ contains all even cycles. To unify the case 
%$m=3,$ we use the following weaker version, which is sufficient for our proof.

\begin{lemma}[Entringer, Schmeichel~\cite{EntringerSchmeichel}]\label{lem:ES}
Let $G$ be a hamiltonian bipartite graph of order $2m$, where $m>3$. If $e(G) > \frac{m^2}{2}$, then $G$ contains a cycle of every even length $4,\dots,2m$.
\end{lemma}

%For $m\ge 4$ this follows from the Entringer--Schmeichel theorem that an $n$-vertex hamiltonian bipartite graph with $n>6$ and more than $n^2/8$ edges is bipancyclic; the case $m=3$ is elementary. 
We will also need well-known extremal estimates for forbidding short cycles. A classical result (see K\H{o}v\'ari, S\'os and Tur\'an~\cite{Kovari}, Reiman~\cite{Reiman})  provides a general upper bound for \(\operatorname{ex}(n,C_4)\).
 
\begin{lemma}[K\H{o}v\'ari, S\'os, Tur\'an~\cite{Kovari}, Reiman~\cite{Reiman}]\label{lem:C4free}
If $G$ is a $C_4$-free graph of order $n$, then
$e(G) \le \frac{n}{4}\bigl(1+\sqrt{4n-3}\bigr)$.
\end{lemma}

The $C_5$ case relies on the following quantitative theorem of Erd\H{o}s.

\begin{lemma}[Erd\H{o}s~\cite{ErdosCircuits}]\label{lem:Erdos5}
Every graph on $2m$ vertices with at least $m^2+1$ edges contains at least $m(m-1)(m-2)$ cycles of length $5$.
\end{lemma}

The remaining tools are used to handle the short cycles
$C_6,C_7,C_8,C_9$ and $C_{10}$.

\begin{lemma}[Gy\H{o}ri, Kensell, Tompkins~\cite{GyoriKensellTompkins}]\label{lem:GKT}
Every $C_6$-free graph $G$ contains a bipartite $C_4$-free subgraph with at least $\frac{3e(G)}{8}$ edges.
\end{lemma}

\begin{lemma}[Wang, Wang~\cite{WangWang}]\label{lem:WangWang}
Let $k,r,n$ be integers with $k\ge 2$, $3\le r\le 2k$, and $n \ge 2(r+2)(r+1)(r+2k)$.
If an $n$-vertex $C_{2k+1}$-free graph $G$ satisfies
$e(G) \ge \floor{\frac{(n-r+1)^2}{4}} + \binom{r}{2}$,
then $G$ contains no odd cycle of length greater than $r$.
\end{lemma}

\begin{lemma}[Naor, Verstra\"ete~\cite{NaorVerstraete}]\label{lem:NV}
Let $a,b,k$ be integers with $a\le b$ and $k\ge 2$. Then
\[
  \operatorname{ex}(a,b,C_{2k}) \le
  \begin{cases}
    (2k-3)\bigl((ab)^{\frac{k+1}{2k}}+a+b\bigr), & \text{if $k$ is odd},\\[2mm]
    (2k-3)\bigl(a^{\frac{k+2}{2k}}b^{\frac{1}{2}}+a+b\bigr), & \text{if $k$ is even}.
  \end{cases}
\]
\end{lemma}

\section{Proof of Theorem~\ref{thm:main440}}\label{sec3}

We prove the following slightly stronger result.

\begin{theorem}\label{thmm}
Let $G$ be a hamiltonian graph of order $n$ satisfying \eqref{eq1}. Then, for every integer  $\ell$ with $3\le \ell\le n$,  $c_\ell(G) \ge n-\ell+2$ if \\
(i) $\ell=3,4,5$ or $\ell\ge 11$ for any $n$, \\
(ii) $\ell=6$, and $n\le 10$ or $n\ge 39$, \\
(iii) $\ell=7$, and $n\le 12$ or $n\ge 360$, \\
(iv) $\ell=8$, and $n\le 14$ or $n\ge 97$,\\
(v) $\ell=9$, and $n\le 16$ or $n\ge 440$,\\
(vi) $\ell=10$, and $n\le 18$ or $n\ge 124$.
%Moreover, any counterexample to Conjecture~\ref{conj:Hilton} is contained in the following finite list of parameter ranges:
%\[
%(6,11\le n\le 38),\quad (7,13\le n\le 359),\quad (8,15\le n\le 96),
%\]
%\[
%(9,17\le n\le 439),\quad (10,19\le n\le 123).
%\]
\end{theorem}

The cases $\ell = 3,5,n$ follow from known results in the literature, so we take these cases as our starting point.

\begin{lemma}\label{ell}
Theorem~\ref{thmm} holds for $\ell=3,5,n$. 
\end{lemma}
\begin{proof}
If $\ell=n$ or $\ell=3$, then the result follows from Lemmas~\ref{lem:SheehanErdos}~and~\ref{lem:Erdos} respectively. 

Now let $\ell=5$. The case $n=5$ is the known case $\ell=n$. Assume that $n\ge 6$.
If $n=2m$, then $m\ge 3$ and \eqref{eq1} becomes $e(G)\ge m^2+2$. 
By Lemma~\ref{lem:Erdos5},  $c_5(G)\ge  m(m-1)(m-2)>m-1+m-2=2m-3 = n-3$, as desired.
If $n=2m+1\ge 7$, then $m\ge 3$ and \eqref{eq1} becomes $e(G) \ge m(m+1)+2$. 
Let $H$ be a spanning subgraph of $G$ with exactly $m(m+1)+2$ edges. 
Then there is a vertex $u$ of $H$ whose degree is at most $m+1$. Note that $H-u$ is a graph on $2m$ vertices with at least $m(m+1)+2-(m+1)=m^2+1$ edges. It then follows from Lemma~\ref{lem:Erdos5} again that $H-u$ has at least $m(m-1)(m-2)$ cycles of length $5$, and so  $c_5(G)\ge m(m-1)(m-2)>n-3$, as desired.
\end{proof}

Having established the lemma, we now proceed to the main result.

\begin{proof}[Proof of Theorem~\ref{thmm}]
Let $G$ be a hamiltonian graph of order $n$ satisfying \eqref{eq1}. 
Our proof proceeds by contradiction. Suppose to the contrary that $c_\ell(G) \le n-\ell+1$. By Lemma~\ref{ell}, $\ell\le n-1$ and $\ell=4$ or $\ell\ge 6$. 

Firstly, we establish the following fundamental lemma. Part of the lemma is straightforward, but the crucial and nontrivial part is that $H$ cannot be bipartite.

\begin{lemma}\label{lem:delete}
There is a spanning hamiltonian subgraph $H$ of $G$ such that $H$ is $C_\ell$-free and
\begin{equation}\label{eq:deletebound}
e(H) \ge \floor{\frac{n^2}{4}} - n + \ell + 1.
\end{equation}
Moreover, $H$ cannot be bipartite. 
\end{lemma}

\begin{proof}
Let $C$ be a Hamilton cycle of $G$. As $\ell\le n-1$, every $\ell$-cycle  contains at least one edge outside $C$. 
For each $\ell$-cycle choose one of its edges outside $C$. Let $E_1$ be the set of chosen edges, where it is possible  that the same edge is chosen for several cycles. Then $|E_1|\le c_\ell(G)\le n-\ell+1$. Let $H$ be the subgraph of $G$ obtained by removing the edges of $E_1$. Then 
 $H$ is spanning and still contains the Hamilton cycle $C$.

By the construction of $H$, every $\ell$-cycle of $G$ has lost one of its edges, so $H$ is $C_\ell$-free. Moreover, by \eqref{eq1}, we have
\[
e(H) =e(G)-|E_1|\ge e(G) - (n-\ell+1) \ge \floor{\frac{n^2}{4}} + 2 - (n-\ell+1)= \floor{\frac{n^2}{4}} - n + \ell + 1,
\]
proving \eqref{eq:deletebound}.

Next, we show that $H$ cannot be bipartite.
Toward a contradiction, suppose that $H$ is bipartite. As $H$ is a spanning hamiltonian subgraph of $G$, $n=2m$ for some integer $m$.

\begin{claim}\label{lem:bipeven} %lem?
$\ell$ is odd.
\end{claim}
\begin{proof}
Suppose to the contrary that $\ell$ is even. Then $\ell=2r$ for some integer $r\ge 2$, so \[
e(H) \ge \floor{\frac{n^2}{4}} - n + \ell + 1=m^2 - 2m + 2r + 1.
\] 
Note that $r\ge 2$ and $m\ge 3$. If $m=3$, then $e(H)\ge 8$, so apart from a Hamilton cycle, there are at least two chords, and from the bipartiteness, it  can not be $C_4$-free. This shows that $m>3$. Then 
\[
e(H)\ge m^2 - 2m + 2r + 1>\frac{m^2}{2}
\]
%We have
%\[
%m^2 - 2m + 2r + 1-\frac{m(m+1)}{2}=  \frac{m^2 - 5m + 4r + 2}{2}=\frac{(m-4)(m-1)+4r-2}{2}>0,
%\]
%and so $e(H)>\frac{m(m+1)}{2}$, 
contradicting Lemma~\ref{lem:ES}.
\end{proof}

By Claim~\ref{lem:bipeven},  $\ell=2r+1$ for some integer $r$ with  $3\le r\le m-1$. 

%The proof is a counting argument. \label{lem:bipodd}

Let $(X,Y)$ be the bipartition of $H$. As $H$ is hamiltonian and bipartite,  $|X|=|Y|=m$. 
As $H$ is bipartite, all edges of $G$ inside $X$ or inside $Y$ must have been deleted to form $H$.
Then $G=B+I$ for some $B$ with $H\subseteq B\subseteq K_{m,m}$ and $I=E(G[X])\cup E(G[Y])$.
Let 
\[
h=m-r,\, e(B)=m^2-t \mbox{ and }|I|=p.
\]
Note that $I\subseteq E(G)\setminus E(H)$. 
From $G$ to $H$, we have deleted at most $n-\ell+1=2m-(2r+1)+1=2h$ edges. Then $p\le 2h$.  
On the other hand, from \eqref{eq1}, we have 
\[
m^2 - t + p=e(B)+p =e(G)\ge m^2+2, 
\]
and hence $p\ge t+2$. So $t\le p-2\le 2h-2$. 

Let $uv\in I$. Assume that $u,v\in X$. In $K_{m,m}$, a $(u,v)$-path of length $2r$ contains $r$ vertices of $Y$ and $r-1$ vertices of $X\setminus \{u,v\}$, and so  
the number of $(u,v)$-paths of length $2r$ is $$T:= (m)_r\,(m-2)_{r-1}.$$ Among these paths, fix an edge $f \in E(K_{m,m}) \setminus E(B)$ that is incident to $u$ or $v$, and denote the other endpoint of $f$ by $w$. Then $w\in Y$. 
Such a path containing $f$ contains $r-1$ vertices of $Y \setminus \{w\}$ and $r-1$ vertices of $X \setminus \{u,v\}$,  
so there are at most
$$ (m-1)_{r-1}(m-2)_{r-1} = \frac{T}{m} $$
$(u,v)$-paths of length $2r$ containing a fixed edge in $E(K_{m,m}) \setminus E(B)$ that is incident to $u$ or $v$.
Now fix an edge $xy\in E(K_{m,m})\setminus E(B)$ with $x\in X\setminus \{u,v\}$ and $y\in Y$. The edge $xy$ can occupy any one of the $2(r-1)$ internal edge positions of the path. Once this position is fixed, the remaining $r-1$ vertices of $Y$ can be chosen and ordered in at most $(m-1)_{r-1}$ ways, and the remaining $r-2$ vertices of $X\setminus \{u,v,x\}$ can be chosen and ordered in at most $(m-3)_{r-2}$ ways. Hence at most
\[
2(r-1)(m-1)_{r-1}(m-3)_{r-2} =\frac{2(r-1)}{m(m-2)}T
\]
$(u,v)$-paths of length $2r$ contain a fixed edge in $E(K_{m,m})\setminus E(B)$ that is incident to neither $u$ nor $v$.

\noindent
{\bf Case 1.} $m\le 2r$. 

In this case, $\frac{2(r-1)}{m(m-2)}T\ge \frac{T}{m}$. Recall that $e(B)=m^2-t$. Then at most
\[
t\frac{2(r-1)}{m(m-2)}T\le (2h-2)\frac{2(r-1)}{m(m-2)}T= \frac{4(m-r-1)(r-1)}{m(m-2)}T\le \left(1-\frac{2}{m}\right)T 
\]
$(u,v)$-paths of length $2r$ contain an edge in $E(K_{m,m})\setminus E(B)$, where the last inequality follows  $4(m-r-1)(r-1)\le \big((m-r-1)+(r-1)\big)^2=(m-2)^2$.
So at least $\frac{2T}{m}$ $(u,v)$-paths of length $2r$ remain in $B$. Each remaining path together with the edge $uv$ gives a distinct $(2r+1)$-cycle of $G$.
Finally, as $r\ge 3$ and $m\ge 4$, we have
\[
\frac{2T}{m} = 2(m-1)_{r-1}(m-2)_{r-1} \ge 2(m-1)(m-2)^2 > 2m-5 \ge 2h+1,
\]
implying that $c_{\ell}(G)\ge \frac{2T}{m}>2h+1=n-\ell+2$, a contradiction. 

\noindent
{\bf Case 2.} $m\ge 2r+1$. 

In this case, $\frac{2(r-1)}{m(m-2)}\le\frac{1}{m}$. At most $\frac{t}{m}T$ $(u,v)$-paths of length $2r$ contain an edge in $E(K_{m,m})\setminus E(B)$, and so at least $T-\frac{t}{m}T$ $(u,v)$-paths of length $2r$ remain in $B$.
Note that an edge $uv\in I$ together with a $(u,v)$-path of length $2r$ forms a $(2r+1)$-cycle. Now taking the sum over all $uv\in I$, we have $c_{2r+1}(G)\ge p\left(1-\frac{t}{m}\right)T$. 

If $t\le m-1$, then $(t+2)(m-t)-(m+1)=(t+1)(m-t-1)\ge 0$. As $p\ge t+2$ and $T\ge m(m-2)$, we have
\[
c_{2r+1}(G)\ge p\left(1-\frac{t}{m}\right)T \ge (t+2)(m-t)(m-2)\ge (m+1)(m-2) > 2h+1=n-\ell+2,
\]
a contradiction. So $t\ge m$. 
Combining this with $t\le 2h-2$, we have $m\le 2h-2=2m-2r-2$, so $m\ge 2r+2$. 

Let $s = \floor{\frac{m}{2}}$. Let $\binom{Y}{s}$  denote 
the set of all $s$-subsets of $Y$. For $Y'\in \binom{Y}{s}$, let $E(Y')=E_{K_{m,m}}(X,Y')\setminus E(B)$. 
Let $a_X=\min\{|E(Y')|:Y'\in \binom{Y}{s}\}$ and $Y_0\in \binom{Y}{s}$ satisfying $a_X=|E(Y_0)|$. 
For each edge $xy\in E(K_{m,m})\setminus E(B)$ with $y\in Y$, 
\[
|\{Y': xy\in E(Y')\}|=\binom{m-1}{s-1}.
\]
%the edge $xy$ belongs to $E(Y')$ for exactly $\binom{m-1}{s-1}$ sets $Y'\in\mathcal{Y}$. 
Hence $\sum_{Y'\in \binom{Y}{s}}|E(Y')| = t\binom{m-1}{s-1}$, and so 
%Note that $|\mathcal{Y}|=\binom{m}{s}$. Then 
$$a_X\le \frac{t\binom{m-1}{s-1}}{\binom{m}{s}}=\frac{st}{m}\le \frac{t}{2}\le h-1.$$
It follows that  at least $m-a_X\ge m-(h-1)=r+1$ vertices of $X$ are adjacent to all vertices of $Y_0$ in $B$. Let $X_0\subseteq X$ be a subset with $|X_0|=r+1$ containing such vertices. Then $B[X_0,Y_0]\cong K_{r+1,s}$.   

If there are at least two vertices in $X$ having at most one neighbor in $Y_0$ in $B$, then at least $s-1$ vertices of $Y_0$ are not neighbors of these vertices in $B$, so $a_X\ge 2(s-1)$, implying  $h-1\ge 2s-2$. However, as $r\ge 3$, $2s-2\ge m-3> m-r-1=h-1$, a contradiction. So there is at most one vertex in $X$ with at most one neighbor in $Y_0$ in $B$. Denote this vertex by $x^*$ if it exists, and fix $x^*\in X\setminus X_0$ otherwise. Then $x^*\notin X_0$, and every vertex of $X\setminus\{x^*\}$ has at least two neighbors in $Y_0$.

\begin{claim}\label{xstar}
Each edge in $G[X]$ is incident with $x^*$. %$y^*$. 
\end{claim}
\begin{proof}
Suppose to the contrary that $uv\in E(G[X])$ and $u,v\neq x^*$. Then both $u$ and $v$ have at least two neighbors in $Y_0$. 

Suppose first that $r\ge 4$.
We count the number of $(u,v)$-paths of length $2r$ in $G$.  First note that for pairs of neighbors of $u$ and $v$, there are at least $2$ choices. Since $B[X_0,Y_0]\cong K_{r+1,s}$, we only consider the case where the internal vertices of the required paths lie in $X_0\cup Y_0$. The order of vertices in $X_0$ can be chosen in at least $(r-1)!$ ways, and there are   at least $(s-2)_{r-2}$ ways for the choices of vertices in $Y_0$. So the number of such  paths is at least $2(r-1)!(s-2)_{r-2}$. Each such path together with $uv$  gives a $(2r+1)$-cycle. As $m\ge 2r+2$, $s\ge r+1\ge 5$ and so \[
2(r-1)!(s-2)_{r-2}\ge 12(s-2)(s-3)\ge 4s-5\ge 2h+1=n-\ell+2,
\]
implying that $c_{\ell}(G)\ge n-\ell+2$, a contradiction. 

Suppose next that $r=3$. Then $m\ge 8$ and hence $s\ge 4$. We again count $(u,v)$-paths of length $6$ whose internal vertices lie in $X_0\cup Y_0$, each of which together with $uv$ gives a $7$-cycle. Let $\eta=|\{u,v\}\cap X_0|$.

If $\eta=0$, then $u,v\notin X_0$. Since both $u$ and $v$ have at least two neighbors in $Y_0$, there are at least $2$ ordered pairs $(y_1,y_3)$ with $y_1\in N_B(u)\cap Y_0$, $y_3\in N_B(v)\cap Y_0$, and $y_1\neq y_3$. After choosing such $y_1,y_3$, we choose an ordered pair $(x_1,x_2)$ of distinct vertices of $X_0$ in $4\cdot 3=12$ ways, and then choose $y_2\in Y_0\setminus\{y_1,y_3\}$ in $s-2$ ways. Thus there are at least $24(s-2)$ paths of the form $uy_1x_1y_2x_2y_3v$.

If $\eta=1$, say $u\notin X_0$ and $v\in X_0$, then we first choose $y_1\in N_B(u)\cap Y_0$, giving at least $2$ choices. Next we choose an ordered pair $(x_1,x_2)$ of distinct vertices of $X_0\setminus\{v\}$, giving $3\cdot 2=6$ choices, and finally an ordered pair $(y_2,y_3)$ of distinct vertices of $Y_0\setminus\{y_1\}$, giving $(s-1)(s-2)$ choices. Hence there are at least $12(s-1)(s-2)$ paths of the form $uy_1x_1y_2x_2y_3v$. 

If $\eta=2$, then $u,v\in X_0$. We choose an ordered pair $(x_1,x_2)$ of the two vertices in $X_0\setminus\{u,v\}$, which gives $2$ choices, and then an ordered triple $(y_1,y_2,y_3)$ of distinct vertices of $Y_0$, which gives $s(s-1)(s-2)$ choices. Hence there are at least $2s(s-1)(s-2)$ paths of the form $uy_1x_1y_2x_2y_3v$.

Since $m\le 2s+1$, we have $2h+1=2m-5\le 4s-3$. As $s\ge 4$, each of the three lower bounds above is greater than $4s-3$, and hence greater than $2h+1$. Therefore $c_\ell(G)\ge 2h+1=n-\ell+2$, again a contradiction. 
\end{proof}

By symmetry, interchanging the roles of $X$ and $Y$, there is a vertex $y^*\in Y$ such that every vertex of $Y\setminus\{y^*\}$ has at least two neighbors in a corresponding $s$-subset of $X$, say $X_0'$. By the same argument as in the proof of Claim~\ref{xstar}, each edge in $I\cap G[Y]$ is incident with $y^*$.  %If no exceptional vertex exists on that side, choose $y^*$ outside the corresponding complete bipartite core.
Thus $I$ is contained in the union of two stars. 
Let $q_X = |E(G[X])|$ and $q_Y = |E(G[Y])|$.
As $q_X+q_Y=p$, we have $q_X+q_Y\ge t+2\ge m+2$. Note that each star contains at most $m-1$ edges. Then $q_X\ge 3$ and $q_Y\ge 3$. 

\begin{claim}\label{xy0}
$N_B(x^*)\cap Y_0= \emptyset$, and $N_B(y^*)\cap X_0'=\emptyset$.
\end{claim}
\begin{proof}
We prove the first statement, and the proof of the second is the same. 
Suppose to the contrary that $y_0\in N_B(x^*)\cap Y_0$. For each $u\in N_G(x^*)\cap X$, choose $y_u\in N_B(u)\cap Y_0$ with $y_u\ne y_0$. We count the number of $(y_u,y_0)$-paths of length $2r-2$ in which all internal vertices are contained in $(X_0\cup Y_0)\setminus \{u\}$. There are at least $r!$ choices for the vertices from $X_0\setminus \{u\}$ and $(s-2)_{r-2}$ choices for the vertices from $Y_0$. Hence there are at least $r!(s-2)_{r-2}$ such paths. Each such path together with the edges $x^*u$, $uy_u$ and $y_0x^*$ gives a  $(2r+1)$-cycle. Since $q_X\ge 3$, this gives at least $3r!(s-2)_{r-2}$ cycles of length $2r+1$. Note that \[
3r!(s-2)_{r-2}\ge \begin{cases}
18(s-2)>4s-3\ge 2h+1&\mbox{ if } r=3,\\
72(s-2)(s-3)>4s-5\ge 2h+1&\mbox{ if }r\ge 4.
\end{cases}
\]
Then $c_{\ell}(G)\ge 2h+1=n-\ell+2$, a contradiction. 
\end{proof}

Let $D$ be a Hamilton cycle of $B$. Let $y_L,y_R\in Y$ be the two neighbors of $x^*$ on $D$, and let $x_L,x_R\in X$ be the neighbors on $D$ of $y_L,y_R$, respectively, other than $x^*$. By Claim~\ref{xy0}, $y_L,y_R\notin Y_0$. %and $x_L,x_R\notin X_0'$. 
As $x^*$ is the only possible vertex of $X$ with at most one neighbor in $Y_0$, both $x_L$ and $x_R$ have at least two neighbors in $Y_0$. 
Let $u\in N_G(x^*)\cap (X\setminus \{x_L,x_R\})$. 
Choose distinct vertices $y,z\in Y_0$ such that $y\in N_B(u)$ and $z\in N_B(x_L)$. Let $Q_L$ be a $(y,z)$-path of length $2r-4$ whose internal vertices are contained in $(X_0\cup Y_0)\setminus \{u,x_L\}$. Then $x^*uyQ_Lzx_Ly_Lx^*$ is a cycle of length $2r+1$ passing $x_L$ and $y_L$. Similarly, we obtain a cycle of length $2r+1$ passing $x_R$ and $y_R$.  
Summing up all $u\in N_G(x^*)\cap (X\setminus \{x_L,x_R\})$, we obtain at least  $2(q_X-2)$ cycles of length $2r+1$. 
%(it is possible $x_L,x_R\not\in N_G(x^*)) 
By similar argument as above, we obtain at least $2(q_Y-2)$ cycles of length $2r+1$. Note that these cycles are pairwise distinct. So \[
c_{\ell}(G) \ge 2(q_X-2)+2(q_Y-2)=2p-8\ge 2(m+2)-8 = 2m-4 > 2m-2r+1 = 2h+1,
\]
a contradiction. This completes the proof of the lemma.
\end{proof}

By Lemma~\ref{lem:delete}, let $H$ be a non-bipartite spanning hamiltonian $C_\ell$-free graph satisfying \eqref{eq:deletebound}.

\begin{claim}\label{largeell} 
It is impossible that $\ell\ge 11$. 
\end{claim}
\begin{proof}
Suppose to the contrary that $\ell\ge 11$. Then \eqref{eq:deletebound} gives
$e(H) \ge \floor{\frac{n^2}{4}} - n + 12$.
By Lemma~\ref{lem:HeHu}, $H$ is weakly pancyclic. As $H$ is $C_\ell$-free, weak pancyclicity forces $g(H) > \ell$.
In particular $H$ is $C_4$-free. As $n>\ell\ge 11$, we have $n\ge 12$, and so by Lemma~\ref{lem:A1},
\[
\floor{\frac{n^2}{4}} - n + 12 >  \floor{\frac{n^2}{4}} - n + 5 > \frac{n}{4}\bigl(1+\sqrt{4n-3}\bigr),
\]
contradicting Lemma~\ref{lem:C4free}.
\end{proof}

\begin{claim}\label{ell4}
$\ell\in \{6,7,8,9,10\}$. 
\end{claim}
\begin{proof} Suppose to the contrary that this is not true. 
By Claim~\ref{largeell}, $\ell=4$.  Then \eqref{eq:deletebound} gives $
e(H) \ge \floor{\frac{n^2}{4}} - n + 5$. As $H$ is $C_4$-free, we have by Lemma~\ref{lem:C4free} that $e(H)\le \frac{n}{4}\bigl(1+\sqrt{4n-3}\bigr)$. If $n\ge 8$, then by Lemma~\ref{lem:A1}, we have $e(H)\le \frac{n}{4}\bigl(1+\sqrt{4n-3}\bigr)<\floor{\frac{n^2}{4}} - n + 5$, a contradiction. This shows that $n=5,6,7$. 

Let $C:=v_0v_1\dots v_{n-1}v_0$ be a Hamilton cycle of $H$ with  indices read modulo $n$.

If $n=5$, adding one chord to $C$  results in a graph that contains a $4$-cycle.
 Hence $H=C$ and $e(H)=5$, contradicting the fact that $e(H)\ge \floor{\frac{n^2}{4}} - n + 5=6$.

Suppose that $n=6$. A chord joining opposite vertices of $C$ immediately creates two  $4$-cycles, and any two distinct distance-two chords create a $4$-cycle. 
Thus at most one chord is present and $e(H)\le 7$, which contradicts the fact that $e(H)\ge \floor{\frac{n^2}{4}} - n + 5=8$.

Suppose that $n=7$. A chord of cyclic distance three creates a $4$-cycle, so if  there is a chord, it  must have cyclic distance two. 
Write such a chord as $a_i=v_iv_{i+2}$. If there are two chords $a_i$ and $a_{i+1}$, then
$v_iv_{i+1}v_{i+3}v_{i+2}v_i$
is a $4$-cycle. If there are three chords $a_i, a_j, a_k$ with $i<j<k$, $j\ne i+1$, $k\ne j+1$, then, $j-i=k-j=2$ or $\{j-i, k-j\}=\{2,3\}$. So we can assume 
%up to rotation and reversal, they are 
$\{i,j,k\}=\{0,2,4\}$ or $\{0,2,5\}$. In the first case $v_0v_2v_4v_6v_0$ is a $4$-cycle, and in the second case $v_0v_2v_4v_5v_0$ is a $4$-cycle. Hence at most two chords are present and $e(H)\le 9$, contradicting the fact that $e(H)\ge \floor{\frac{n^2}{4}} - n + 5=10$. This proves the claim. 
\end{proof}

By Claim \ref{ell4},  $\ell\in\{6,7,8,9,10\}$. 

\begin{claim}\label{nlarge}
$n\ge 2\ell-1$.
\end{claim}
\begin{proof}
Suppose to the contrary that $n \le 2\ell-2$. 
Indeed, if $n=2m$, then $\ell\ge m+1$ and if $n=2m+1$, then  $\ell\ge m+2$.
So 
\[
  \floor{\frac{n^2}{4}}-n+\ell+1\ge \floor{\frac{(n-1)^2}{4}}+2.
\]
By \eqref{eq:deletebound}, we have $e(H)\ge \floor{\frac{(n-1)^2}{4}}+2$. Then by Lemma~\ref{lem:HFS}, as $H$ is non-bipartite, $H$ is pancyclic, which contradicts the fact that $H$ is $C_\ell$-free.
\end{proof}

\begin{claim}\label{ell6}
If $\ell=6$, then $11\le n\le 38$. 
\end{claim}
\begin{proof}
By Lemma~\ref{lem:GKT}, $H$ has a bipartite $C_4$-free subgraph $F$ with $e(F)\ge \frac{3e(H)}{8}$. By \eqref{eq:deletebound} and Lemma~\ref{lem:C4free},
\[
\frac{3}{8}\left(\floor{\frac{n^2}{4}} - n + 7\right) \le e(F) \le \frac{n}{4}\bigl(1+\sqrt{4n-3}\bigr).
\]
By Lemma~\ref{lem:A2}, we have $ n\le 38$. Then the result follows by Claim~\ref{nlarge}. 
\end{proof}

The following observation will be useful in the rest of the proof.

\begin{lemma}\label{lem:oddgt3}
Every hamiltonian non-bipartite graph of order at least $6$ contains an odd cycle of length greater than $3$.
\end{lemma}

\begin{proof}
Let $C$ be a Hamilton cycle. If the order $n$ is odd, then $C$ itself is an odd cycle of length $n>3$. If $n$ is even, color the vertices alternately along $C$. Since the graph is non-bipartite, some chord has both endpoints in the same color class. Equivalently, its endpoints have even distance $d$ along $C$, where $2\le d\le \frac{n}{2}$. The chord and the two paths of $C$ between its endpoints form odd cycles of lengths $d+1$ and $n-d+1$. If one of these lengths is $3$, the other is $n-1>3$, otherwise both are greater than $3$.
\end{proof}

\begin{claim}\label{ell7}
If $\ell=7$, then $13\le n\le 359$. If $\ell=9$, then $17\le n\le 439$. 
\end{claim}
\begin{proof}
The lower bound for $n$ follows by Claim~\ref{nlarge}.
Let $\ell=2k+1$ with $k=3$ or $4$. 

For $\ell=7$, Lemma~\ref{lem:delete} gives 
\[
e(H) \ge \floor{\frac{n^2}{4}} - n + 8 = \floor{\frac{(n-2)^2}{4}} + 7>\floor{\frac{(n-2)^2}{4}} + 3.
\]
If $n \ge 2(3+2)(3+1)(3+2\cdot 3) = 360$, then applying Lemma~\ref{lem:WangWang} with $r = 3$, $H$ contains no odd cycle of length greater than $3$, contradicting Lemma~\ref{lem:oddgt3}. So $n\le 359$.

For $\ell=9$, Lemma~\ref{lem:delete} gives 
\[
e(H) \ge \floor{\frac{n^2}{4}} - n + 10 = \floor{\frac{(n-2)^2}{4}} + 9>\floor{\frac{(n-2)^2}{4}} + 3.
\]
If $  n \ge 2(3+2)(3+1)(3+2\cdot 4)=440$, then applying Lemma~\ref{lem:WangWang} with $r = 3$  yields a contradiction to Lemma~\ref{lem:oddgt3}. So $n\le 439$.
\end{proof}

\begin{claim}\label{ell8}
If $\ell=8$, then $15\le n\le 96$. If $\ell=10$, then $19\le n\le 123$. 
\end{claim}
\begin{proof}
The lower bound for $n$ follows by Claim~\ref{nlarge}.
Let $\ell=2k$ with $k=4$ or $5$. 
Let $F$ be a bipartite spanning subgraph of $H$ containing at least half of the edges of $H$ \cite{Erd3}.
Then
\[
  e(F) \ge \frac12 e(H) \ge \frac12\left(\floor{\frac{n^2}{4}} - n + 2k + 1\right),
\]
and $F$ is $C_{2k}$-free. Let the part sizes of $F$ be $a$ and $b$ with $a\le b$, so $a+b=n$.

If $\ell=8$, then Lemma~\ref{lem:NV} gives $e(F) \le 5\bigl(a^{\frac{3}{4}}b^{\frac{1}{2}}+a+b\bigr)\le 5\bigl((\frac{n}{2})^{\frac{5}{4}}+n\bigr)$. By Lemma~\ref{lem:A3}, if $n\ge 97$, then $e(F)\le 5\bigl((\frac{n}{2})^{\frac{5}{4}}+n\bigr)<  \frac{1}{2}\left(\floor{\frac{n^2}{4}} - n + 9\right)$, contradicting the lower bound on $e(F)$. Therefore, $n\le 96$.

If $\ell=10$, then Lemma~\ref{lem:NV} gives
$e(F) \le 7\bigl((ab)^{\frac{3}{5}}+a+b\bigr) \le 7\bigl((\frac{n}{2})^{\frac{6}{5}}+n\bigr)$. By Lemma~\ref{lem:A3}, if $n\ge 124$, then $e(F)\le 7\bigl((\frac{n}{2})^{\frac{6}{5}}+n\bigr)< \frac12\left(\floor{\frac{n^2}{4}} - n + 11\right)$, again a contradiction. Therefore, $n\le 123$.
\end{proof}

Now, (i) follows from Claim~\ref{ell4}, (ii) follows from Claim~\ref{ell6}, (iii) and (v) follow from Claim~\ref{ell7}, and (iv) and (vi) follow from  Claim~\ref{ell8}. % proving the theorem. 
\end{proof}

\section{Remarks}

Sheehan \cite{Sheehan} explicitly states that there is a hamiltonian  graph $G$  on $n\ge 5$ vertices with 
\[
e(G) = \left\lfloor \frac{n^2}{4} \right\rfloor + 2
\]
and exactly two Hamilton cycles. Thus, for $\ell=n$, the lower bound  $n-\ell+2 = 2$ is attained, and the conjecture is tight for $\ell=n\ge 5$.
The endpoint \(\ell=3\) is also sharp for infinitely many orders. Let \(n=2m+1\ge 5\), and let \(G\) be obtained from \(K_{m,m+1}\) by adding two edges inside the part of size \(m+1\). Then \(G\) is hamiltonian,
\[
e(G)=m(m+1)+2=\left\lfloor \frac{n^2}{4}\right\rfloor+2,
\]
and every triangle consists of one of the two added edges together with one vertex in the other part. Hence \(c_3(G)=2m=n-1\), attaining the bound for \(\ell=3\).

\appendix

\section{Appendix}

For completeness, we prove the inequalities used in Section~\ref{sec3}.

\begin{lemma}\label{lem:A1}
For every integer $n\ge 8$,
$\floor{\frac{n^2}{4}} - n + 5 > \frac{n}{4}\bigl(1+\sqrt{4n-3}\bigr)$.
\end{lemma}
\begin{proof}
Suppose first that $n$ is even, i.e., $n=2m$ for some $m\ge 4$. The desired inequality is equivalent to $2m^2-5m+10 > m\sqrt{8m-3}$.
Let
\[
P(m)=(2m^2-5m+10)^2 - m^2(8m-3).
\]
Then $P(m)=4m^4-28m^3+68m^2-100m+100$,  $P'(m) = 4(4m^3 - 21m^2 + 34m - 25)$ and $P''(m)=8(6m^2-21m+17)$. As $m\ge 4$, $P''(m)>0$, $P'(m)$ is increasing on  $[4,\infty)$. Note that $P'(4)=124>0$. Then $P'(m)>0$, implying that $P(m)$ is also increasing on $[4, \infty)$. Again, as $P(4)=20>0$, we have $P(m)>0$ for all $m\ge 4$.  So the inequality holds.   

Suppose next that $n$ is odd, i.e., $n=2m+1$ for some $m\ge 4$. The desired inequality is equivalent to $4m^2-6m+15 > (2m+1)\sqrt{8m+1}$.
Let
\[
Q(m)=(4m^2-6m+15)^2 - (2m+1)^2(8m+1).
\]
Then $Q(m)=16m^4-80m^3+120m^2-192m+224$, $Q'(m) = 16(4m^3 - 15m^2 + 15m - 12)$ and $Q''(m) = 48(4m^2 - 10m + 5)$. As $m\ge 4$, $Q''(m)>0$ and so $Q'(m)$ is increasing on $[4, \infty)$. Note that $Q'(4)=1024>0$. Then $Q'(m)>0$, implying that $Q(m)$ is also increasing on $[4,\infty)$. Again, as $Q(4)=352>0$, we have $Q(m)>0$ for all $m\ge 4$.  So the inequality holds.    
\end{proof}

\begin{lemma}\label{lem:A2}
For every integer $n\ge 39$,
$\frac38\left(\floor{\frac{n^2}{4}} - n + 7\right) > \frac{n}{4}\bigl(1+\sqrt{4n-3}\bigr)$.
\end{lemma}

\begin{proof}
Suppose first that $n$ is even, i.e., $n=2m$ for some $m\ge 20$. The desired inequality is equivalent to $3(m^2-2m+7) > 4m\bigl(1+\sqrt{8m-3}\bigr)$, and is also equivalent to $3m^2-10m+21>4m\sqrt{8m-3}$.
Let
\[
P(m)=(3m^2-10m+21)^2 - 16m^2(8m-3).
\]
Then $P(m)=9m^4-188m^3+274m^2-420m+441$, $P'(m) = 4(9m^3 - 141m^2 + 137m - 105)$ and $P''(m) = 4(27m^2 - 282m + 137)$. As $m\ge 20$, $P''(m)>0$, $P'(m)$ is increasing on $[20,\infty)$. Note that $P'(20)=72940>0$. Then $P'(m)>0$, implying that $P(m)$ is also increasing on $[20, \infty)$. Again, as $P(20)=37641>0$, we have $P(m)>0$ for all $m\ge 20$.  So the inequality holds.   

Suppose next that $n$ is odd, i.e., $n=2m+1$ for some $m\ge 19$, and the desired inequality is equivalent to
$3(m^2-m+6) > 2(2m+1)\bigl(1+\sqrt{8m+1}\bigr)$, and is also equivalent to $3m^2-7m+16> 2(2m+1)\sqrt{8m+1}$. 
Let
\[
Q(m)=(3m^2-7m+16)^2 - 4(2m+1)^2(8m+1).
\]
Then $Q(m)=9m^4-170m^3+m^2-272m+252$, $Q'(m) = 36m^3 - 510m^2 + 2m - 272$ and $Q''(m) = 108m^2 - 1020m + 2$. As $m\ge 19$, $Q''(m)>0$, so $Q'(m)$ is increasing on $[19,\infty)$. Note that $Q'(19)=62580>0$. Then $Q'(m)>0$, implying that $Q(m)$ is also increasing on $[19,\infty)$. Again, as $Q(19)=2304>0$, we have $Q(m)>0$ for all $m\ge 19$.  So the inequality holds.    
\end{proof}

\begin{lemma}\label{lem:A3}
(i) For every integer $n\ge 97$,
$\frac12\left(\floor{\frac{n^2}{4}} - n + 9\right) > 5\bigl((\frac{n}{2})^{5/4}+n\bigr)$.\\
(ii) For every integer $n\ge 124$,
$\frac12\left(\floor{\frac{n^2}{4}} - n + 11\right) > 7\bigl((\frac{n}{2})^{6/5}+n\bigr)$.
\end{lemma}

\begin{proof}
Since $\floor{\frac{n^2}{4}}\ge \frac{n^2-1}{4}$, it is enough to prove positivity of the continuous functions
\[
  f(x)=\frac{x^2-1}{8} - \frac{x}{2} + \frac92 - 5\bigl(\bigl(\frac{x}{2}\bigr)^{5/4}+x\bigr)
\]
for $x\ge 97$, and
\[
  g(x)=\frac{x^2-1}{8} - \frac{x}{2} + \frac{11}{2} - 7\bigl(\bigl(\frac{x}{2}\bigr)^{6/5}+x\bigr)
\]
for $x\ge 124$. By a direct calculation, we have 
$f(97)>7$ and $g(124)>6$.
Moreover,
$f'(x)=\frac{x}{4}-\frac{11}{2}-\frac{25}{8}(\frac{x}{2})^{1/4}$ and $f'(97)>10$, while $f''(x)=\frac14 - \frac{25}{64}(\frac{x}{2})^{-3/4} > 0$ for $x\ge 97$.
Hence $f(x)$ is increasing on $[97,\infty)$ and is positive. Similarly,
$g'(x)=\frac{x}{4}-\frac{15}{2}-\frac{21}{5}(\frac{x}{2})^{1/5}$ and $g'(124)>13$, while
$g''(x)=\frac14 - \frac{21}{50}(\frac{x}{2})^{-4/5} > 0$ for $x\ge 124$.
Thus $g(x)$ is increasing on $[124,\infty)$ and is positive. 
\end{proof}

\medskip

\noindent {\bf Acknowledgements.} 
This work was supported by the
National Natural Science Foundation of China (No.~12571364).

\end{document}